\documentclass[a4paper,11pt]{amsart}

\usepackage{amsthm,amsfonts,amsmath,amssymb,verbatim,rotating} 
\usepackage{tikz,graphicx}
\usepackage{mathrsfs}
\usepackage{braket}

\usepackage{fouridx}

\usepackage[normalem]{ulem}

\usepackage{hyperref}

\numberwithin{equation}{section}
\newtheorem{theorem}{Theorem}[section]
\newtheorem{corollary}[theorem]{Corollary}
\newtheorem{proposition}[theorem]{Proposition}
\newtheorem{lemma}[theorem]{Lemma}

\renewcommand{\leq}{\leqslant}
\renewcommand{\geq}{\geqslant}

\newcommand{\dup}{\mathrm{d}}

\newcommand{\iup}{\hspace{1pt}\mathrm{i}\hspace{1pt}}
\newcommand{\Int}{\int\limits}

\newcommand{\abs}[1]{\lvert#1\rvert}
\newcommand{\la}{\lambda}
\newcommand{\La}{\Lambda}
\newcommand{\Part}{\mathscr{P}}
\newcommand{\twocore}[1]{2\textup{-core}(#1)}
\renewcommand{\H}{\mathcal{H}}
\newcommand{\Ho}[1]{\mathcal{H}^{\mathrm{o}}_{#1}}
\newcommand{\He}[1]{\mathcal{H}^{\mathrm{e}}_{#1}}
\DeclareMathOperator*{\pf}{Pf}

\begin{document}
\title[Littlewood identities]
{Proof of some Littlewood identities conjectured by Lee, Rains and Warnaar}
\author{Seamus P. Albion}
\address{Fakult\"{a}t f\"{u}r Mathematik, Universit\"{a}t Wien, 
Oskar-Morgenstern-Platz 1, A-1090 Vienna, Austria}
\email{seamus.albion@univie.ac.at}



\begin{abstract}
We prove a novel pair of Littlewood identities for Schur functions,
recently conjectured by Lee, Rains and Warnaar in the Macdonald case, 
in which the sum is over partitions with empty 2-core. 
As a byproduct we obtain a new Littlewood identity in the spirit of 
Littlewood's original formulae.
\end{abstract}

\date{}
\maketitle

\section{Introduction}
The classical Littlewood identities are the following three summation 
formulae for Schur functions:
\begin{subequations}\label{Eq_classical}
\begin{gather}\label{Eq_classical1}
\sum_{\la} s_\la(x)=\prod_{i\geq 1}\frac{1}{1-x_i}
\prod_{i<j}\frac{1}{1-x_ix_j}, \\
\sum_{\substack{\la\\\text{$\la$ even}}}s_\la(x)
=\prod_{i\geq 1}\frac{1}{1-x_i^2}\prod_{i<j}\frac{1}{1-x_ix_j},
\label{Eq_classical2} \\
\sum_{\substack{\la\\\text{$\la'$ even}}}s_\la(x)
=\prod_{i<j}\frac{1}{1-x_ix_j}, \label{Eq_classical3}
\end{gather} 
\end{subequations}
where $x=(x_1,x_2,x_3,\dots)$ is a countable alphabet.
Here and throughout the rest of the paper 
``$\la$ even'' means the partition $\la$ has only even
parts and $\la'$ denotes the conjugate of $\la$.
These identities were first written down together by Littlewood 
\cite[p.~238]{Littlewood40},
however \eqref{Eq_classical1} was already known to Schur
\cite{Schur18}.
They have since afforded many far-reaching 
generalisations and have found applications in areas such as 
combinatorics, representation theory and elliptic hypergeometric series.
In particular there are many generalisations of \eqref{Eq_classical}
at the Schur level \cite{Bressoud00,Des86,HKKO23,IW99,JZ01,Kawanaka99,Okada98,
Proctor90,Stembridge90}.
Also see \cite{RW21} for comprehensive references to the literature.

The purpose of this note is to prove the Schur 
function case of a pair of Littlewood identities for Macdonald polynomials 
recently conjectured by Lee, Rains and Warnaar \cite[Conjecture~9.5]{LRW20}.
To state these we need some notation. 
Denote the multiset of hook lengths of a partition $\la$ by $\H_\la$.
We refine this by writing $\H^{\mathrm{e}/\mathrm{o}}_\la$ 
for the submultiset of even/odd hook lengths.
The standard infinite $q$-shifted factorial is given by
$(a;q)_\infty:=\prod_{i\geq 0}(1-aq^i)$ and we define a statistic
\begin{equation}\label{Eq_varsigma}
\varsigma(\la):=\sum_{(i,j)\in\la}(-1)^{\la_i+\la_j'-i-j+1}(\la_i-i),
\end{equation}
in terms of the Young diagram of $\la$; see Subsection~\ref{Sec_parts} below.
Finally, let $\hat\La_\mathbb{F}$ denote the completion of the ring of 
symmetric functions over the field $\mathbb{F}$ with respect to the natural 
grading by degree.

\begin{theorem}\label{Thm_q=t}
As identities in $\hat\La_{\mathbb{Q}(q)}$ at the alphabet
$x=(x_1,x_2,x_3,\dots)$ we have that
\begin{equation}\label{Eq_L1}
\sum_{\substack{\la\\\twocore{\la}=0}}
q^{\varsigma(\la)}\frac{\prod_{h\in\Ho{\la}}(1-q^h)}{\prod_{h\in\He{\la}}(1-q^h)}
s_\la(x)=\prod_{i\geq 1}\frac{(qx_i^2;q^2)_\infty}{(x_i^2;q^2)_\infty}
\prod_{i<j}\frac{1}{1-x_ix_j},
\end{equation}
and
\begin{equation}\label{Eq_L2}
\sum_{\substack{\la\\\twocore{\la}=0}}
q^{\varsigma(\la')}
\frac{\prod_{h\in\Ho{\la}}(1-q^h)}{\prod_{h\in\He{\la}}(1-q^h)}s_\la(x)=
\prod_{i\geq 1}\frac{(q^2x_i^2;q^2)_\infty}{(qx_i^2;q^2)_\infty}
\prod_{i<j}\frac{1}{1-x_ix_j}.
\end{equation}
\end{theorem}

The condition $\twocore{\la}=0$ generalises both the even row and even column 
conditions of \eqref{Eq_classical2} and \eqref{Eq_classical3}.
Indeed, by Lemma~\ref{Lem_even-rows} we have that $\varsigma(\la)=0$ if and only if 
$\la$ is even.
Thus when setting $q=0$ \eqref{Eq_L1} and \eqref{Eq_L2} collapse to
\eqref{Eq_classical2} and \eqref{Eq_classical3} respectively.
In this sense these identities are in the spirit of Kawanaka's 
identity \cite[Theorem~1.1]{Kawanaka99}
\[
\sum_\la \prod_{h\in\mathcal{H}_\la}\bigg(\frac{1+q^h}{1-q^h}\bigg)s_\la(x)
=\prod_{i\geq 1}\frac{(-qx_i;q)_\infty}{(x_i;q)_\infty}
\prod_{i<j}\frac{1}{1-x_ix_j},
\]
since this reduces to \eqref{Eq_classical1} when $q=0$.
Unlike Kawanaka's identity one can make sense of the $q\to1$ limit of
\eqref{Eq_L1} and \eqref{Eq_L2}. In either case we obtain the following
Littlewood-type identity.
\begin{corollary}\label{Cor}
As an identity in $\hat\La_\mathbb{Q}$ at the alphabet 
$x=(x_1,x_2,x_3,\dots)$,
\[
\sum_{\substack{\la\\\twocore{\la}=0}}
\frac{\prod_{h\in\Ho{\la}}h}{\prod_{h\in\He{\la}}h}s_\la(x)=
\prod_{i\geq 1}\frac{1}{(1-x_i^2)^{1/2}}
\prod_{i<j}\frac{1}{1-x_ix_j}.
\]
\end{corollary}

The outline of the paper is as follows.
In the next section we give preliminaries regarding partitions, Schur
functions and Koornwinder polynomials.
In Section~\ref{Sec_integral} we prove a pair of vanishing integrals for
Schur functions again conjectured by Lee, Rains and Warnaar in the Macdonald
case \cite[Conjecture~9.2]{LRW20}.
Then, in Section~\ref{Sec_proofs}, we follow the techniques of
\cite{RW21} to prove the bounded analogues of Theorem~\ref{Thm_q=t}
conjectured in \cite[Conjecture~9.4]{LRW20}.
The theorem then follows by taking an appropriate limit.
We conclude with a derivation of Corollary~\ref{Cor}.

\section{Partitions and ($\mathrm{BC}_n$)-symmetric functions}\label{Sec_pre}
\subsection{Partitions}\label{Sec_parts}
A partition $\la=(\la_1,\la_2,\la_3,\dots)$ is a weakly decreasing sequence
of nonnegative integers such that finitely many $\la_i$ are nonzero.
The sum of the entries is denoted $\abs{\la}:=\la_1+\la_2+\la_3+\cdots$
and if $\abs{\la}=n$ we say $\la$ is a partition of $n$.
Nonzero entries are called parts, and the number of parts is called the length,
denoted $l(\la)$.
We denote by $\Part$ the set of all partitions and by $\Part_n$ the set of 
all partitions with length at most $n$.
In particular $\Part_0=\{0\}$ where $0$ denotes the unique partition of zero.
If $\la\in\Part_n$ we write $\la+\delta$ for the partition
$(\la_1+n-1,\la_2+n-2,\dots,\la_n)$.
The number $m_i(\la)$ of occurrences of an integer $i$ as a part of $\la$ is 
called the multiplicity.
Sometimes we express a partition in terms of its multiplicities as
$\la=(1^{m_1(\la)}2^{m_2(\la)}3^{m_3(\la)}\cdots)$.
We write $\mu\subset\la$ if the partition $\mu$ is contained in $\la$,
i.e.~if $\mu_i\leq \la_i$ for all $i\geq 1$.
If $\la\subseteq(m^n)$ for some nonnegative integers $m,n$, then
we write $(m^n)-\la$ for the complement of $\la$ inside $(m^n)$, that is,
$(m^n)-\la:=(m-\la_n,m-\la_{n-1},\dots,m-\la_1)$.
A partition is identified with its Young diagram, which is the
left-justified array of squares with $\la_i$ squares in 
row $i$ with $i$ increasing downward.
For example
\smallskip
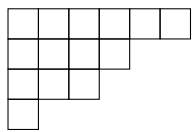
\begin{center}
\begin{tikzpicture}[scale=0.4]
\foreach \i [count=\ii] in {6,4,3,1}
\foreach \j in {1,...,\i}{\draw (\j,1-\ii) rectangle (\j+1,-\ii);}
\end{tikzpicture}
\end{center}
is the Young diagram of $(6,4,3,1)$.
The conjugate of a partition, written $\la'$, is obtained by reflecting
the Young diagram in the main diagonal, so that $(6,4,3,1)'=(4,3,3,2,1,1)$.
The arm and leg lengths of a square $s=(i,j)\in\la$ are given by
\[
a(s):=\la_i-j \quad\text{and}\quad l(s):=\la_j'-i,
\]
which are the number of boxes strictly to the right and below $s$ respectively.
The hook length is the sum of these including $s$ itself, so that
$h(s):=a(s)+l(s)+1$.
Using the same example as above, in the Young diagram
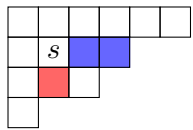
\smallskip
\begin{center}
\begin{tikzpicture}[scale=0.4]
\filldraw[color=blue!60] (3,-1) rectangle (5,-2);
\filldraw[color=red!60] (2,-2) rectangle (3,-3);
\foreach \i [count=\ii] in {6,4,3,1}
\foreach \j in {1,...,\i}{\draw (\j,1-\ii) rectangle (\j+1,-\ii);}
\draw (2.5,-1.5) node {$s$};
\end{tikzpicture}
\end{center}
we have labelled the square $s=(2,2)$ so that $a(s)=2$, $l(s)=1$ and
$h(s)=4$.
As in the introduction we denote the multiset of hook lengths of $\la$ by
$\mathcal{H}_\la$. This is further refined as $\He{\la}$ and $\Ho{\la}$, 
the multisets of hook lengths which are even or odd, respectively.
In terms of these we define the hook polynomials
\begin{align*}
H_\la(q)&:=\prod_{h\in\mathcal{H}_\la}(1-q^h) \\
H_\la^{\mathrm{e}/\mathrm{o}}(q)&:=\prod_{h\in\mathcal{H}^{
\mathrm{e}/\mathrm{o}}_\la}(1-q^h),
\end{align*}
which are invariant under conjugation of $\la$.
For $z\in\mathbb{C}$ we also need the content polynomials
\begin{align*}
C_\la(z;q)&:=\prod_{(i,j)\in\la}(1-zq^{j-i}) \\
C_\la^{\mathrm{e}/\mathrm{o}}(z;q)&:=\prod_{
\substack{(i,j)\in\la\\\text{$i+j$ even/odd}}}(1-zq^{j-i}).
\end{align*}

In this paper we will frequently encounter partitions with empty $2$-core,
written $\twocore{\la}=0$.
One definition of such partitions is that their diagrams may be tiled by
dominoes. Our running example of $(6,4,3,1)$ has empty $2$-core since 
it admits the tiling 
\smallskip
\begin{center}
\begin{tikzpicture}[scale=0.4]
\foreach \i [count=\ii] in {6,4,3,1}
\foreach \j in {1,...,\i}{\draw (\j,1-\ii) rectangle (\j+1,-\ii);}
\draw[ultra thick,color=blue!60] (1.5,-0.5) -- (1.5,-1.5);
\draw[ultra thick,color=blue!60] (1.5,-2.5) -- (1.5,-3.5);
\draw[ultra thick,color=blue!60] (2.5,-1.5) -- (2.5,-2.5);
\draw[ultra thick,color=blue!60] (3.5,-1.5) -- (3.5,-2.5);
\draw[ultra thick,color=blue!60] (2.5,-0.5) -- (3.5,-0.5);
\draw[ultra thick,color=blue!60] (4.5,-0.5) -- (4.5,-1.5);
\draw[ultra thick,color=blue!60] (5.5,-0.5) -- (6.5,-0.5);
\end{tikzpicture}
\end{center}
which is clearly not unique.
We will now give some conditions which are equivalent to $\la$ having empty
$2$-core which all easily follow by induction on $\abs{\la}$.
The reader interested in more general statements involving Littlewood's 
decomposition of a partition into its $r$-core and $r$-quotient for all
$r\geq 2$ may consult, for example, \cite{Littlewood51} 
or \cite[p.~12--15]{Macdonald95}.
\begin{lemma}\label{Lem_2core}
For $\la\in\Part_{2n}$ the following are equivalent:
\begin{enumerate}
\item $\twocore{\la}=0$.
\item $\abs{\Ho{\la}}=\abs{\He{\la}}=n$.
\item The set
\[
\{\la_1+2n-1,\la_2+2n-2,\dots,\la_{2n-1}+1,\la_{2n}\}
\]
contains $n$ even and $n$ odd integers.
\end{enumerate}
\end{lemma}

\subsection{Auxiliary results}
Here we prove some properties of the statistic $\varsigma(\la)$ 
\eqref{Eq_varsigma}.
Firstly, as we have already used in the introduction, we have the following
characterisation of the vanishing of $\varsigma(\la)$.
\begin{lemma}\label{Lem_even-rows}
Let $\twocore{\la}=0$. Then $\varsigma(\la)\geq 0$ with $\varsigma(\la)=0$ if and 
only if $\la$ is even.
\end{lemma}
\begin{proof}
If $\la$ is even then $\varsigma(\la)=0$ since the number of even and odd hook 
lengths in each row is equal. Assume that $\la$ is not even. Then $\la$ has an 
even number of odd parts. Let $\la_{i_1}$, $\la_{i_2}$ be the final two odd 
rows of $\la$. Since $\twocore{\la}$ is empty these must be separated by an 
even number of even rows (possibly zero).
Ignoring the rows above, the contribution to $\varsigma(\la)$ below and
including row
$\la_{i_1}$ may be computed as
\[
\la_{i_1}-\la_{i_2}+i_2-i_1+2\sum_{j=i_1+1}^{i_2-1}(-1)^{i_1+j-1}(\la_j-j).
\]
Since the numbers $\la_j-j$ are strictly decreasing this sum is positive.
The next nonzero contribution to $\varsigma(\la)$ will come from the pair of
odd rows above in the same fashion.
Thus repeating the above shows that if $\la$ has empty $2$-core and contains at 
least two odd rows then $\varsigma(\la)>0$.
\end{proof}
Note that $\varsigma((2,1,1,1))=0$, so that $\varsigma(\la)$ may vanish for partitions with 
nonempty $2$-core. 

\begin{lemma}\label{Lem_n1}
For $\la\in\mathscr{P}_{2n}$ there holds
\begin{equation}\label{Eq_alternativeForm}
\varsigma(\la)=\sum_{(i,j)\in\la+\delta}(-1)^{\la_i-i-j+1}(\la_i-i)-
\sum_{1\leq i<j\leq 2n}(-1)^{\la_i-\la_j+j-i}(\la_i-i).
\end{equation}
Moreover, if $\twocore{\la}=0$, then
\begin{equation}\label{Eq_conjForm}
\varsigma(\la')=
\frac{\abs{\la}}{2}-n^2-n+\sum_{1\leq i<j\leq 2n}
(-1)^{\la_i-\la_j+j-i}(\la_j-j).
\end{equation}
\end{lemma}
\begin{proof}
We interpret the definition of $\varsigma(\la)$ as a sum over the Young 
diagram of 
$\la$ where each square has weight $(-1)^{\la_i+\la_j'-i-j+1}(\la_i-i)$.
In the Young diagram of $\la+\delta$ place the integer 
$(-1)^{\la_i-i-j+1}(\la_i-i)$ in box $(i,j)$.
Summing over $i,j$ gives the first sum on the right of 
\eqref{Eq_alternativeForm}.
To identify the second sum, we remove the columns with index $\la_j+2n-j+1$
for $2\leq j\leq 2n$ whose entries are
$(-1)^{\la_i-\la_j+j-i}(\la_i-i)$. 
The remaining diagram is that of $\la$ with entries
$(-1)^{\la_i+\la_j'-i-j+1}(\la_i-i)$, which shows the first identity.

The proof of the second identity is similar.
Note that by \eqref{Eq_varsigma}, $\varsigma(\la')$ may be written as
\[
\varsigma(\la')=\sum_{(i,j)\in\la}(-1)^{\la_i+\la_j'-i-j+1}(\la_j'-j).
\]
We thus fill the diagram of $\la+\delta$ with integers
$(-1)^{\la_i-i-j+1}(2n-j)$, so that removing the same columns as before now
gives
\[
\varsigma(\la')=\sum_{(i,j)\in\la+\delta}(-1)^{\la_i-i-j+1}(2n-j)
-\sum_{1\leq i<j\leq 2n}(-1)^{\la_i-\la_j+j-i}(j-\la_j-1).
\]
A simple calculation shows that for $\twocore{\la}=0$,
\[
\sum_{(i,j)\in\la+\delta}(-1)^{\la_i-i-j+1}(2n-j)+
\sum_{1\leq i<j\leq 2n}(-1)^{\la_i-\la_j+j-i}=\frac{\abs{\la}}{2}-n^2-n,
\]
completing the proof.
\end{proof}

\subsection{Schur functions}
For completeness we give a definition of the Schur functions in terms
of the classical ratio of alternants.
For $\la\in\Part_n$ the Schur function is defined as
\[
s_\la(x_1,\dots,x_n):=\frac{\det_{1\leq i,j\leq n}(x_i^{\la_j+n-j})}
{\det_{1\leq i,j\leq n}(x_i^{n-j})},
\]
and $s_\la(x_1,\dots,x_n):=0$ for $l(\la)>n$.
The set of the $s_\la(x_1,\dots,x_n)$ indexed over $\Part_n$ forms a $\mathbb{Z}$-basis 
for the ring of symmetric functions in $n$ variables, denoted $\La_n$.
We also use the Schur functions in countably many variables
$x=(x_1,x_2,x_3,\dots)$, such as in Theorem~\ref{Thm_q=t},
which may be defined by the Jacobi--Trudi determinant \cite[p.~41]{Macdonald95}.
The set of such $s_\la(x)$ when indexed over all partitions $\la$ form a
$\mathbb{Z}$-basis for the ring of symmetric functions $\La$.
We also require the ring $\hat\La$ which is the completion of $\La$
with respect to the natural grading by degree \cite[p.~66]{Rains05}.

Several of the results we need below are best stated in terms of Macdonald
polynomials, which are a $q,t$-analogue of the Schur functions 
\cite[\S VI]{Macdonald95}.
We simply note that the Macdonald polynomials $P_\la(x;q,t)$ are a basis for 
$\La_{\mathbb{Q}(q,t)}$ and reduce to the Schur functions 
when $q=t$, i.e., $P_\la(x;q,q)=s_\la(x)$.

\subsection{Koornwinder polynomials and integrals}

The Koornwinder polynomials are a family of $\mathrm{BC}_n$-symmetric 
functions depending on six parameters first introduced by Koornwinder
\cite{Koornwinder92} as a multivariate analogue of the Askey--Wilson
polynomials \cite{AW85}.
Here we write $x=(x_1,\dots,x_n),
x^{\pm}=(x_1,x_1^{-1},\dots,x_n,x_n^{-1})$ and for a 
single-variable function $g(x_i)$ we set
\begin{align*}
g\big(x_i^\pm\big)&:=g(x_i)g\big(x_i^{-1}\big) \\
g\big(x_i^\pm x_j^\pm\big)&:=g(x_ix_j)g\big(x_i^{-1}x_j\big)
g\big(x_ix_j^{-1}\big)g\big(x_i^{-1}x_j^{-1}\big).
\end{align*}
Below the function will be one of $g(x_i)=(x_i;q)_\infty$ or $g(x_i)=(1-x_i)$.
Also for the infinite $q$-shifted factorial we adopt the usual multiplicative
notation
\[
(a_1,\dots,a_n;q)_\infty:=(a_1;q)_\infty\cdots (a_n;q)_\infty.
\]

Let $W:=\mathfrak{S}_n \ltimes (\mathbb{Z}/2\mathbb{Z})^n$ be the group of 
signed permutations on $n$ letters.  
A Laurent polynomial $f(x)\in\mathbb{C}[x^\pm]$ 
is called $\mathrm{BC}_n$-symmetric if it is 
invariant under the natural action of $W$ on the $n$ variables where
the reflections act by $x_i\mapsto 1/x_i$.
For $\la\in\mathscr{P}_n$ define the orbit-sum indexed by $\la$ as
\[
m_\la^{\mathrm{BC}}(x):=\sum_\alpha x^\alpha,
\]
where the sum is over all elements $\alpha$ of the $W$-orbit of $\la$,
the reflections act on sequences by $\alpha_i\mapsto-\alpha_i$, and
$x^\alpha:=x_1^{\alpha_1}\cdots x_n^{\alpha_n}$.
The orbit-sums form a basis for the ring $\La_n^{\mathrm{BC}}$ of 
$\mathrm{BC}_n$-symmetric functions.
For $q,t,t_0,t_1,t_2,t_3\in\mathbb{C}$ with
$\abs{q},\abs{t},\abs{t_0},\abs{t_1},\abs{t_2},\abs{t_3}<1$, 
define the Koornwinder density by
\[
\Delta(x;q,t;t_0,t_1,t_2,t_3)
:=\prod_{i=1}^n\frac{(x_i^{\pm 2};q)_\infty}
{\prod_{r=0}^3(t_rx_i^{\pm};q)_\infty}
\prod_{1\leq i<j\leq n}\frac{(x_i^{\pm}x_j^{\pm};q)_\infty}
{(tx_i^{\pm}x_j^{\pm};q)_\infty}.
\]
This further allows one to define an inner product on $\La_n^{\mathrm{BC}}$
by
\[
\langle f,g\rangle_{q,t;t_0,t_1,t_2,t_3}^{(n)}
:=\Int_{\mathbb{T}^n} f(x)g(x^{-1})\Delta(x;q,t;t_0,t_1,t_2,t_3)\,\dup T(x),
\]
where $\mathbb{T}^n$ is the standard $n$-torus and 
the measure $T(x)$ is given by
\[
\dup T(x):=\frac{1}{2^nn!(2\pi\iup)^n}\frac{\dup x_1}{x_1}\cdots 
\frac{\dup x_n}{x_n}.
\]
The Koornwinder polynomials are defined to be the unique 
$\mathrm{BC}_n$-symmetric functions satisfying
\[
K_\la = m_\la^{\mathrm{BC}} + \sum_{\mu<\la} c_{\la \mu} m_\mu^{\mathrm{BC}},
\]
where $c_{\la \mu} \in \mathbb{C}(q,t,t_0,t_1,t_2,t_3)$, and for which
\[
\langle K_\la,K_\mu\rangle_{q,t;t_0,t_1,t_2,t_3}^{(n)} = 0
\qquad\text{if $\la\neq \mu$}.
\]
Note that $\mu\leq \la$ denotes the extension of the usual dominance order
to all partitions $\la,\mu\in\Part$:
$\mu\leq\la$ if and only if $\mu_1+\cdots+\mu_i\leq\la_1+\cdots+\la_i$ for all
$i\geq 1$.
The Koornwinder polynomials satisfy many nice properties such as 
the quadratic norm evaluation and evaluation symmetry 
\cite{vanDiejen96,Sahi99}.
The key identity we need is \cite[Equation~(2.6.9)]{RW21} (see 
also \cite[Corollary~7.2.1]{Rains05})
\begin{multline}\label{Eq_m-limit}
\lim_{m\to \infty} (x_1\dots x_n)^m
K_{(m^n)-\la}(x;q,t;t_0,t_1,t_2,t_3) \\
=P_\la(x;q,t)
\prod_{i=1}^n\frac{(t_0x_i,t_1x_i,t_2x_i,t_3x_i;q)_\infty}{(x_i^2;q)_\infty}
\prod_{1\leq i<j\leq n}\frac{(tx_ix_j;q)_\infty}{(x_ix_j;q)_\infty}.
\end{multline}
We will only use this for $\la=0$, in which case $P_0(x;q,t)=1$.

For a basis $\{f_\la\}$ of $\La_n^{\mathrm{BC}}$
we write $[f_\la]g$ for the coefficient of $f_\la$ in the expansion
$g=\sum_\la c_\la f_\la$ where the $c_\la$ lie in some coefficient ring.
The virtual Koornwinder integral of a $\mathrm{BC}_n$-symmetric function
$f$ is defined as
\[
I^{(n)}_K(f;q,t;t_0,t_1,t_2,t_3):=[K_0(x;q,t;t_0,t_1,t_2,t_3)]f.
\]
This is extended to allow for symmetric function arguments via the homomorphism
$\La_{2n}\longrightarrow\La_n^{\mathrm{BC}}$ for which
$f(x_1,\dots,x_{2n})\mapsto f(x_1,x_1^{-1},\dots,x_n,x_n^{-1})$.
Of course since $K_0=1$ the orthogonality of the Koornwinder polynomials allows
us to express this as
\[
I^{(n)}_K(f;q,t;t_0,t_1,t_2,t_3)
=\frac{\langle f, 1\rangle_{q,t;t_0,t_1,t_2,t_3}^{(n)}}
{\langle 1,1\rangle_{q,t;t_0,t_1,t_2,t_3}^{(n)}}.
\]
Note that the denominator has the explicit evaluation
\[
\langle 1,1\rangle_{q,t;t_0,t_1,t_2,t_3}^{(n)}
=\prod_{i=1}^n\frac{(t,t_0t_1t_2t_3t^{n+i-2};q)_\infty}
{(q,t^i;q)_\infty\prod_{0\leq r<s\leq 3}(t_rt_st^{i-1};q)_\infty},
\]
which is Gustafson's generalised Askey--Wilson integral \cite{Gustafson90}.
The virtual Koornwinder integral can be evaluated for many choices of
the argument $f$, see \cite{LRW20,Rains05,RV07,RW21}.
In particular, the vanishing integrals of the next section may be 
expressed in terms of virtual Koornwinder integrals.
We need one final identity involving virtual Koornwinder integrals.
To state this conveniently, let
\[
f^{(m)}_\la(q,t;t_0,t_1,t_2,t_3)
:=[P_\la(x;q,t)](x_1\cdots x_n)^m K_{(m^n)}(x;q,t;t_0,t_1,t_2,t_3).
\]
\begin{proposition}[{\cite[Proposition~4.9]{RW21}}]\label{Prop_coef}
For nonnegative integers $n,m$ and $\la\subseteq (2m)^n$,
\[
f_\la^{(m)}(q,t;t_0,t_1,t_2,t_3)=(-1)^{\abs{\la}}
I_K^{(m)}\big(P_{\la'}(t,q);t,q;t_0,t_1,t_2,t_3\big).
\]
\end{proposition}

\section{Vanishing integrals}\label{Sec_integral}
In this section we evaluate a pair of vanishing integrals for Schur functions
conjectured by Lee, Rains and Warnaar in the Macdonald case 
\cite[Conjecture~9.2]{LRW20}.

For $a,b,q\in\mathbb{C}$ with $\abs{a},\abs{b},\abs{q}<1$ we define
\begin{multline*}
I^{(n)}_\la(a,b;q)
:=\frac{1}{Z_n(a,b;q)}
\Int_{\mathbb{T}^n}s_\la\big(x_1^{\pm},\dots,x_n^{\pm}\big)
\prod_{i=1}^n
\frac{(x_i^{\pm2};q)_\infty}{(ax_i^{\pm2},bx_i^{\pm2};q^2)_\infty} \\
\times
\prod_{1\leq i<j\leq n}\big(1-x_i^{\pm}x_j^{\pm}\big)\, \dup T(x),
\end{multline*}
where $\la$ is a partition with length at most $2n$ and the normalising
factor is given by
\begin{align*}
Z_n(a,b;q)&:=
\Int_{\mathbb{T}^n}
\prod_{i=1}^n\frac{(x_i^{\pm2};q)_\infty}{(ax_i^{\pm2},bx_i^{\pm2};q^2)_\infty}
\prod_{1\leq i<j\leq n}\big(1-x_i^{\pm}x_j^{\pm}\big)\, \dup T(x) \\
&\hphantom{:}=\prod_{i=1}^n\frac{(abq^{n+i-2};q)_\infty}
{(q^i,-aq^{i-1},-bq^{i-1};q)_\infty(abq^{2i-2};q^2)_\infty^2}.
\end{align*}
Note that in terms of virtual Koornwinder integrals this is
\[
I_\la^{(n)}(a,b;q)=I^{(n)}_K(s_\la;q,q,a^{1/2},-a^{1/2},b^{1/2},-b^{1/2}).
\]
Lee, Rains and Warnaar prove the following properties of the above integral.
\begin{proposition}[{\cite[Proposition~9.3]{LRW20}}]
For $a,b,q\in\mathbb{C}$ with $\abs{a},\abs{b},\abs{q}<1$ 
and $\la$ a partition of length at most $2n$ the integral
$I^{(n)}_\la(a,b;q)$ vanishes unless $\twocore{\la}=0$.
Furthermore
\begin{subequations}\label{Eq_Pfaff}
\begin{align}\label{Eq_P1}
I^{(n)}_\la(q,q;q) &=
\prod_{i=1}^n\frac{(1-q^{2i-1})^{2n-2i+1}}{(1-q^{2i})^{2n-2i}} \\
&\quad\times\pf_{1\leq i,j\leq 2n}\bigg(\frac{q^{(\la_i-\la_j+j-i-1)/2}}
{1-q^{\la_i-\la_j+j-i}}\chi(\text{$\la_i-\la_j+j-i$ odd})\bigg), \notag
\end{align}
and
\begin{align}\label{Eq_P2}
I^{(n)}_\la(1,q^2;q)&= 
\frac{1}{2^{n-1}(1+q^n)}
\prod_{i=1}^n\frac{(1-q^{2i-1})^{2n-2i+1}}{(1-q^{2i})^{2n-2i}}\\
&\quad\times
\pf_{1\leq i,j\leq 2n}\bigg(\frac{1+q^{\la_i-\la_j+j-i}}
{1-q^{\la_i-\la_j+j-i}}\chi(\text{$\la_i-\la_j+j-i$ odd})\bigg). \notag
\end{align}
\end{subequations}
\end{proposition}

Lee, Rains and Warnaar also give a conjectural Macdonald polynomial analogue
of this proposition \cite[Conjecture~9.2]{LRW20}.
There the generalisations of \eqref{Eq_Pfaff} are explicit products.
Our next proposition gives the evaluation of the Pfaffians in the previous
proposition, verifying the conjecture of Lee, Rains and Warnaar for $q=t$.

\begin{proposition}
For $\la$ with length at most $2n$ and $2\text{-core}(\la)=0$,
\begin{equation}\label{Eq_int1}
I^{(n)}_\la(q,q;q)=q^{\varsigma(\la')}\frac{C_\la^{\mathrm{e}}(q^{2n};q)
H_\la^{\mathrm{o}}(q)}
{C_\la^{\mathrm{o}}(q^{2n};q)H_\la^{\mathrm{e}}(q)}
\end{equation}
and
\begin{equation}\label{Eq_int2}
I^{(n)}_\la(1,q^2;q)=q^{\varsigma(\la)}\frac{1+q^{n+2\varsigma(\la')-2\varsigma(\la)}}{1+q^n}\,
\frac{C_\la^{\mathrm{e}}(q^{2n};q)H_\la^{\mathrm{o}}(q)}
{C_\la^{\mathrm{o}}(q^{2n};q)H_\la^{\mathrm{e}}(q)}.
\end{equation}
\end{proposition}
\begin{proof}
Since the structure of the Pfaffians is similar, we focus on the second
identity, and evaluate \eqref{Eq_P2}.

Fix a partition $\la\in\mathscr{P}_{2n}$ with empty $2$-core.
Define the set $J\subseteq\{1,\dots,2n\}$ as the collection of integers $j$ 
for which column $j$ has a nonzero entry in the first row, and
set $I:=\{1,\dots,2n\}\setminus J$.
Since $\twocore{\la}=0$ it follows that $\abs{I}=\abs{J}=n$.
The elements of $I$ and $J$ are labeled by $i_k$ and $j_k$ respectively,
where $1\leq k\leq n$ and ordered naturally.
With this established we define the $n\times n$ matrix $M$ with entries
$M_{k,\ell}$ by 
\[
M_{k,\ell}:=
\frac{1+q^{\la_{i_k}-\la_{j_\ell}+j_{\ell}-i_k}}
{1-q^{\la_{i_k}-\la_{j_\ell}+j_{\ell}-i_k}}.
\]
The Pfaffian in \eqref{Eq_P2} may be expressed in terms of the determinant 
of $M$. Indeed, by pushing the rows with indices in $J$ to the right we 
see that
\begin{align*}
&\pf_{1\leq i,j\leq 2n}\bigg(\frac{1+q^{\la_i-\la_j+j-i}}
{1-q^{\la_i-\la_j+j-i}}\chi(\text{$\la_i-\la_j+j-i$ odd})\bigg)
\\&\qquad\qquad\qquad\qquad
=(-1)^{\binom{n}{2}+\sum_{j\in J}j}
\pf\begin{pmatrix}0 & M \\ -M^t & 0\end{pmatrix} \\
&\qquad\qquad\qquad\qquad=(-1)^{\sum_{j\in J}j}\det M.
\end{align*}
The determinant may be evaluated simply by applying the following 
generalisation of Cauchy's double alternant which may be found in 
\cite[Example~3.1; $a=0$]{Chu01}:
\begin{align*}
\det_{1\leq i,j\leq n}\bigg(\frac{bx_i+cy_j}{x_i+y_j}\bigg)
&=(b-c)^{n-1}\bigg(b\prod_{i=1}^nx_i+(-1)^{n-1}c\prod_{i=1}^ny_i\bigg)\\
&\qquad\qquad\times
\frac{\prod_{1\leq i<j\leq n}(x_i-x_j)(y_i-y_j)}{\prod_{i,j=1}^n(x_i+y_j)}.
\end{align*}
We apply this with
$(b,c,x_k,y_\ell)\mapsto(-1,1,q^{\la_{i_k}-i_k},-q^{\la_{j_\ell}-j_\ell})$ for
$1\leq k,\ell\leq n$.
After some elementary manipulations the evaluation may now be expressed as 
\begin{align*}
&I_\la^{(n)}(1,q^2;q)\\
&\quad= \frac{\prod_{i\in I}q^{\la_i-i}+\prod_{j\in J}q^{\la_j-j}}{1+q^n}
\prod_{i=1}^n\frac{(1-q^{2i-1})^{2n-2i+1}}{(1-q^{2i})^{2n-2i}} \\
&\quad\quad\times
\prod_{\substack{1\leq i<j\leq 2n\\\text{$\la_i-\la_j+j-i$ even}}}
\frac{1-q^{\la_i-\la_j+j-i}}{q^{\la_j-j}}
\prod_{\substack{1\leq i<j\leq 2n\\\text{$\la_i-\la_j+j-i$ odd}}}
\frac{q^{\la_j-j}}{1-q^{\la_i-\la_j+j-i}}.
\end{align*}
The terms of the form $1-q^x$ can be simplified thanks to the identity
\cite[p.~10--11]{Macdonald95}
\[
\frac{C_\la(q^{2n};q)}{H_\la(q)}
=\prod_{s\in \la}\frac{1-q^{n+c(s)}}{1-q^{h(s)}}
=\frac{\prod_{1\leq i<j\leq n}1-q^{\la_i-\la_j+j-i}}{\prod_{i=1}^n(q;q)_i},
\]
where $l(\la)\leq n$. Restricting all products to even/odd exponents
implies that
\begin{align*}
&\frac{C_\la^{\mathrm{e}}(q^{2n};q)H_\la^{\mathrm{o}}(q)}
{C_\la^{\mathrm{o}}(q^{2n};q)H_\la^{\mathrm{e}}(q)} \\
&\quad=\prod_{\substack{1\leq i<j\leq 2n\\\text{$\la_i-\la_j+j-i$ even}}}
(1-q^{\la_i-\la_j+j-i})
\prod_{\substack{1\leq i<j\leq 2n\\\text{$\la_i-\la_j+j-i$ odd}}}
\frac{1}{1-q^{\la_i-\la_j+j-i}}\\
&\qquad\quad\times
\prod_{i=1}^n\frac{(1-q^{2i-1})^{2n-2i+1}}{(1-q^{2i})^{2n-2i}}.
\end{align*}
It remains to show that the powers of $q$ agree in the prefactor.
Since
\[
\prod_{i\in I}q^{\la_i-i}+\prod_{j\in J}q^{\la_j-j}
=\prod_{\substack{i=1\\\text{$\la_i-i$ even}}}^{2n}q^{\la_i-i}+
\prod_{\substack{i=1\\\text{$\la_i-i$ odd}}}^{2n}q^{\la_i-i},
\]
this may be reduced to the pair of identities
\begin{equation*}
\varsigma(\la)=\sum_{\substack{i=1\\\text{$\la_i-i$ even}}}^{2n}(\la_i-i)
+\sum_{1\leq i<j\leq 2n}(-1)^{\la_i-\la_j+j-i}(\la_j-j),
\end{equation*}
and
\begin{equation*}
n+2\varsigma(\la')-2\varsigma(\la) =
\sum_{\substack{i=1\\\text{$\la_i-i$ odd}}}^{2n}(\la_i-i)-
\sum_{\substack{i=1\\\text{$\la_i-i$ even}}}^{2n}(\la_i-i).
\end{equation*}
In the first of these write
\begin{align*}
\sum_{\substack{i=1\\\text{$\la_i-i$ even}}}^{2n}(\la_i-i)
&=\sum_{(i,j)\in\la+\delta}(-1)^{\la_i-i-j+1}(\la_i-i)
+\sum_{i=1}^{2n}(\la_i-i) \\
&=\varsigma(\la)+\sum_{1\leq i<j\leq 2n}(-1)^{\la_i-\la_j+j-i}(\la_i-i)
+\sum_{i=1}^{2n}(\la_i-i),
\end{align*}
where in the second equality we have applied \eqref{Eq_alternativeForm}
from Lemma~\ref{Lem_n1}.
Since
\begin{align*}
&\sum_{1\leq i<j\leq 2n}(-1)^{\la_i-\la_j+j-i}(\la_i-i)
+\sum_{1\leq i<j\leq 2n}(-1)^{\la_i-\la_j+j-i}(\la_i-i)
+\sum_{i=1}^{2n}(\la_i-i) \\
&\qquad\qquad=\sum_{i,j=1}^{2n}(-1)^{\la_i-\la_j+j-i}(\la_i-i) \\
&\qquad\qquad=0,
\end{align*}
the first identity follows.
For the second identity, a similar rewriting, now using \eqref{Eq_conjForm}
of Lemma~\ref{Lem_n1}, shows us that
\begin{align*}
&\sum_{\substack{i=1\\\text{$\la_i-i$ odd}}}^{2n}(\la_i-i)-
\sum_{\substack{i=1\\\text{$\la_i-i$ even}}}^{2n}(\la_i-i)\\
&\qquad=-2\sum_{(i,j)\in\la+\delta}(-1)^{\la_i-i-j+1}(\la_i-i)
-\sum_{i=1}^{2n}(\la_i-i) \\
&\qquad=-2\varsigma(\la)-\abs{\la}+2n^2+n
-2\sum_{1\leq i<j\leq 2n}(-1)^{\la_i-\la_j+j-i}(\la_i-i) \\
&\qquad=n+2\varsigma(\la')-2\varsigma(\la). 
\end{align*}
This finishes the evaluation of \eqref{Eq_P2}. 
The evaluation of \eqref{Eq_P1} is almost identical except one directly
applies \eqref{Eq_conjForm} of Lemma~\ref{Lem_n1} to compute the exponent
of $q$ in the prefactor.
\end{proof}

\section{Bounded Littlewood identities}\label{Sec_proofs}
Here we use the integral evaluations of the previous section to prove
a bounded analogue of Theorem~\ref{Thm_q=t}.
This is followed by proofs of the theorem and of Corollary~\ref{Cor}.

\subsection{A bounded analogue of Theorem~\ref{Thm_q=t}}

Bounded Littlewood identities are generalisations of ordinary Littlewood
identities in which the largest part of the indexing partition has an 
upper bound, say $m$, such that sending $m$ to infinity recovers an ordinary
(unbounded) Littlewood identity.
The first example of such an identity was discovered by 
Macdonald \cite[\S1.5]{Macdonald79} where he used a bounded analogue of 
\eqref{Eq_classical1} to prove the MacMahon and Bender--Knuth 
conjectures on plane partitions \cite{BK72,MacMahon99}.
Bounded analogues of the remaining two classical identities 
\eqref{Eq_classical2} and \eqref{Eq_classical3} were obtained by
D\'esarm\'enien, Proctor and Stembridge \cite{Des86,Proctor90,Stembridge90}
and Okada \cite{Okada98} respectively.
A host of other bounded identities for Hall--Littlewood and Macdonald 
polynomials may be found in \cite{RW21} and references therein.
For further discussion of the history of bounded Littlewood identities
see \cite{HKKO23}.
We now state the bounded analogue of Theorem~\ref{Thm_q=t}.
\begin{theorem}\label{Thm_bounded}
For nonnegative integers $m$ and $n$,
\begin{multline}\label{Eq_B1}
\sum_{\substack{\la\\\twocore{\la}=0}} 
q^{\varsigma(\la')}\frac{C_\la^{\mathrm{e}}(q^{-2m};q)
H_\la^{\mathrm{o}}(q)}{C_\la^{\mathrm{o}}(q^{-2m};q)H_\la^{\mathrm{e}}(q)}
s_\la(x) \\
=(x_1\cdots x_n)^m K_{(m^n)}\big(x;q,q;q^{1/2},-q^{1/2},q^{1/2},-q^{1/2}\big),
\end{multline}
and
\begin{multline}\label{Eq_B2}
\sum_{\substack{\la\\\twocore{\la}=0}} 
\frac{q^{2\varsigma(\la')-\varsigma(\la)}+q^{m+\varsigma(\la)}}{1+q^m}\,
\frac{C_\la^{\mathrm{e}}(q^{-2m};q)
H_\la^{\mathrm{o}}(q)}{C_\la^{\mathrm{o}}(q^{-2m};q)H_\la^{\mathrm{e}}(q)}
s_\la(x) \\
=(x_1\cdots x_n)^m K_{(m^n)}(x;q,q;1,-1,q,-q).
\end{multline}
\end{theorem}
These identities are indeed bounded since $C_\la^{\mathrm{e}}(q^{-2m};q)$
vanishes if $\la_1>2m$.
Since, by \cite[Lemma~4.1]{LRW20}, the Koornwinder polynomials on the right
reduce to classical group characters for $q=0$, one recovers the previously 
mentioned D\'esarm\'enien--Proctor--Stembridge and Okada identities 
respectively in this case.
The Koornwinder polynomials for $q=t$ on the right-hand side may alternatively 
be expressed as a ratio of determinants of Askey--Wilson polynomials
\cite{AW85}; see, e.g., \cite[Definition~4.1]{CMW17}.
This, however, does not seem to shed light on a more explicit expression for
the evaluation of these sums.
In particular, the specialisations of $K_{(m^n)}$ above are not contained
in \cite[Lemma~4.1]{LRW20}.

The following argument is sketched in \cite[\S9]{LRW20}, but we give the 
details in the Schur case.
Assuming the Macdonald polynomial version of the vanishing integrals
\cite[Conjecture~9.2]{LRW20}, the same argument gives the conjectural 
Littlewood identities.
\begin{proof}[Proof of Theorem~\ref{Thm_bounded}]
The goal is to find an expression for the coefficient of $s_\la(x)$ in the
Schur expansion of the right-hand side. 
By Proposition~\ref{Prop_coef} this coefficient is
\[
f_\la^{(m)}(x;q,q,t_0,t_1,t_2,t_3)
=(-1)^{\abs{\la}}I_K^{(m)}(s_{\la'}(x);q,q;t_0,t_1,t_2,t_3).
\]
If we specialise $(t_0,t_1,t_2,t_3)=(q^{1/2},-q^{1/2},q^{1/2},-q^{1/2})$
then this reduces to
\[
f_\la^{(m)}\big(x;q,q;q^{1/2},-q^{1/2},q^{1/2},-q^{1/2}\big)
=(-1)^{\abs{\la}}I_{\la'}^{(m)}(q,q;q).
\]
The integral on the right is \eqref{Eq_int1}, as desired, and vanishes 
unless $\twocore{\la}=0$. In this case the sign disappears since $\abs{\la}$ is
even and we obtain 
\[
(-1)^{\abs{\la}}I_{\la'}^{(m)}(q,q;q)
=q^{\varsigma(\la)}\frac{C_{\la'}^{\mathrm{e}}(q^{2m};q)H_{\la'}^{\mathrm{o}}(q)}
{C_{\la'}^{\mathrm{o}}(q^{2m};q)H_{\la'}^{\mathrm{e}}(q)}.
\]
By \cite[Lemma~2.3]{LRW20} we may alternatively express this as
\begin{equation}\label{Eq_conjugate}
q^{\varsigma(\la)}\frac{C_{\la'}^{\mathrm{e}}(q^{2m};q)H_{\la'}^{\mathrm{o}}(q)}
{C_{\la'}^{\mathrm{o}}(q^{2m};q)H_{\la'}^{\mathrm{e}}(q)}
=q^{\varsigma(\la')}\frac{C_\la^{\mathrm{e}}(q^{-2m};q)H_\la^{\mathrm{o}}(q)}
{C_\la^{\mathrm{o}}(q^{-2m};q)H_\la^{\mathrm{e}}(q)}
\end{equation}
This establishes \eqref{Eq_B1}.
For \eqref{Eq_B2} the same procedure applies with the substitution
$(t_0,t_1,t_2,t_3)=(1,-1,q,-q)$ and by using the integral \eqref{Eq_int2}.
\end{proof}

\subsection{Proof of Theorem~\ref{Thm_q=t}}
With the bounded identities established we may take the 
$m\to\infty$ limit of both identities to obtain their unbounded counterparts.
For the Koornwinder side we use \eqref{Eq_m-limit} with $(\la,q,t)=(0,q,q)$ 
and $(t_0,t_1,t_2,t_3)=(q^{1/2},-q^{1/2},q^{1/2},-q^{1/2})$ or 
$(t_0,t_1,t_2,t_3)=(1,-1,q,-q)$. In the case of \eqref{Eq_B1} this yields
\begin{align*}
&\lim_{m\to \infty} (x_1\dots x_n)^m
K_{(m^n)}\big(x;q,q;q^{1/2},-q^{1/2},q^{1/2},-q^{1/2}\big) \\
&\quad=
\prod_{i=1}^n\frac{(q^{1/2}x_i,-q^{1/2}x_i,q^{1/2}x_i,-q^{1/2}x_i;q)_\infty}
{(x_i^2;q)_\infty}
\prod_{1\leq i<j\leq n}\frac{1}{1-x_ix_j} \\
&\quad=\prod_{i=1}^n\frac{(qx_i^2;q^2)_\infty}{(x_i^2;q^2)_\infty}
\prod_{1\leq i<j\leq n}\frac{1}{1-x_ix_j},
\end{align*}
where we have used
\[
(a,-a;q)_\infty=(a^2;q^2)_\infty.
\]
For the limit of the summand we use it in conjugate form \eqref{Eq_conjugate}
so that
\[
\lim_{m\to\infty}
q^{\varsigma(\la)}\frac{C_{\la'}^{\mathrm{e}}(q^{2m};q)H_\la^{\mathrm{o}}(q)}
{C_{\la'}^{\mathrm{o}}(q^{2m};q)H_\la^{\mathrm{e}}(q)}
=q^{\varsigma(\la)}\frac{H_\la^{\mathrm{o}}(q)}{H_\la^{\mathrm{e}}(q)}.
\]
Thus we have proved \eqref{Eq_L1}.
As before the same procedure yields \eqref{Eq_L2}.

\subsection{Proof of Corollary~\ref{Cor}}
In order to obtain Corollary~\ref{Cor} we take $q\to 1$ in either 
\eqref{Eq_L1} or \eqref{Eq_L2}.
Let $(a;q)_n:=\prod_{k=0}^{n-1}(1-aq^k)$.
Then we may take the limit of the product-side of \eqref{Eq_L1} by using
\begin{align*}
\lim_{q\to 1}\frac{(qx_i^2;q^2)_\infty}{(x_i^2;q^2)_\infty}
&=\lim_{q\to 1}\sum_{n=0}^\infty\frac{(q;q^2)_n}{(q^2;q^2)_n}x_i^{2n} \\
&=\sum_{n=0}^\infty\frac{1\cdot 3\cdots (2n-1)}{2\cdot 4\cdots 2n}x_i^{2n} \\
&=\frac{1}{(1-x_i^2)^{1/2}},
\end{align*}
where in the first line we have applied the $q$-binomial theorem 
\cite[Equation~(1.3.2)]{GR04}:
\[
\sum_{n=0}^\infty\frac{(a;q)_n}{(q;q)_n}z^n
=\frac{(az;q)_\infty}{(z;q)_\infty}.
\]
The $q\to1$ limit of the product-side of \eqref{Eq_L2} gives the same result.
The limit of either sum follows from the characterisation of partitions with
empty $2$-core in Lemma~\ref{Lem_2core}, namely that 
$\abs{\He{\la}}=\abs{\Ho{\la}}$.

\subsection*{Acknowledgements}
I thank Christian Krattenthaler and Ole Warnaar for many useful 
discussions and suggestions, and for their encouragement.

\end{document}